\documentclass[twocolumn]{autartarxiv}    

\usepackage{graphicx}          
\usepackage{amsmath} 
\usepackage{amssymb}  
\usepackage{dsfont}

\def\QEDopen{{\setlength{\fboxsep}{0pt}\setlength{\fboxrule}{0.2pt}\fbox{\rule[0pt]{0pt}{1.3ex}\rule[0pt]{1.3ex}{0pt}}}}
\def\QED{\QEDopen}

\def\endproof{\hspace*{\fill}~\QED\par\endtrivlist\unskip}

\def\be{\begin{equation}}
\def\ee{\end{equation}}
\def\ba{\begin{array}}
\def\ea{\end{array}}
\def\eqa{\begin{eqnarray}}
\def\eqe{\end{eqnarray}}

              \newtheorem{problem}{Problem}
              \newtheorem{definition}{Definition}
\newtheorem{rema}{Remark}
\newtheorem{theo}{Theorem}
\newtheorem{lemm}{Lemma}

\begin{document}

\begin{frontmatter}

\title{Quasi-Optimal Regulation of Flow Networks with Input Constraints\thanksref{footnoteinfo}} 

\thanks[footnoteinfo]{This paper was not presented at any IFAC
meeting. Corresponding author T.~W.~Scholten. Tel. +3150 363 3077.}

\author[Groningen]{Tjardo Scholten}\ead{t.w.scholten@rug.nl},    
\author[Groningen]{Claudio De Persis}\ead{c.de.persis@rug.nl},               
\author[Groningen]{Pietro Tesi}\ead{p.tesi@rug.nl}  

\address[Groningen]{Department of ENTEG, Faculty of Mathematics and Natural Sciences, University of Groningen, Nijenborgh 4, 9747 AG Groningen, The Netherlands}  


\begin{abstract}                          
In this work we consider a flow network for which the goal is to solve a practical optimal regulation problem in the presence of input saturation. Based on Lyapunov arguments we propose distributed controllers which guarantee global convergence to an arbitrarily small neighborhood of the desired optimal steady state while fulfilling the constraints. As a case study we apply our distributed controller to a district heating network.
\end{abstract}

\end{frontmatter}

\section{INTRODUCTION}
\label{introduction}

Regulation of interconnected dynamical systems, recently received much attention due to its many different applications, see {\it e.g.} \cite{lovisari2014stability,burbano2014distributed,Bauso20132206}.
Related examples are control of DC networks \cite{zhao2015distributed}, state regulation of heating, ventilation and air conditioning (HVAC) systems \cite{gupta2015distributed}, compartmental flow control \cite{blanchini2016compartmental}, rendezvous and formation control \cite{kim2015adaptation}, pressure regulation in hydraulic networks \cite{DePersis2011}, \cite{DePersis2014}, and frequency synchronization in power grids \cite{trip_2016_automatica}.
The models used in these examples are often similar to the ones used for flow networks, in which the control problem is to regulate the state, by assigning the flows on the links.

The stability of flow networks under time-varying disturbances can be guaranteed by means of internal model based controllers on the edges, as has been shown in \cite{Buerger2015} and \cite{de2013balancing}. It is well know that these controllers can also be implemented in a distributed fashion, such as in \cite{como2013robust}.

Port-Hamiltonian (PH) systems have also proven to be a powerful tool for the modeling and control of nonlinear networked systems \cite{van2014port}.
These PH systems have been used extensively to model physically interconnected dynamical systems as well as to synthesize controllers that ensure output regulation \cite{van2012hamiltonian,romero2013robust,ortega2008control,jayawardhana2007passivity}.

Besides stability, it is often desirable to have optimal flows, according to some cost function. Static optimization problems have been discussed in great detail (see {\it e.g.}  \cite{bertsekas1998network}, \cite{boyd2004}  and  \cite{rockafellar1984network}), and are commonly referred to as the mathematical theory of {\it network optimization}. Moreover, it is useful to solve these optimization problems in a distributed fashion, such as in \cite{gharesifard2014distributed}, in order to avoid excessive communication and computation times.
However, most real networks have to react dynamically to changes in the network, which requires feedback controllers.
This is for example done in \cite{Bauso20132206} where controllers are designed for linear systems that achieve asymptotic optimality. In \cite{burger2014duality} this is extended to non-linear systems by using passivity arguments.

Rather than including an optimality condition on the flow, \cite{trip_2016_automatica} assigns a cost function to the inputs and guarantees optimal state regulation for power networks.
This approach is also combined with optimal flows in \cite{burger2015dynamic}, which additionally considers capacity constraints on the transportation lines.
However, these constraints depend on the initial conditions, which is not desirable in networks that have, {\it e.g.} physical constraints. Accordingly, the motivation arises to consider input and state constraints in regulation problems for flow networks for which the constraints are never violated, independent of the initial conditions.

Model predictive control (MPC) handles input and state constraints in a natural way, as has been shown in \cite{danielson2013constrained}, where a capacity maximization and balancing problem is solved. However, the stability of MPC systems is often hard to analyze and running MPC algorithms is computationally intensive.
A solution that avoids the use of MPC but does not consider any optimality is provided in \cite{burger2013hierarchical}. In this paper they show that there exists a strong relation between clustering, optimal network flow problems and output agreement.

In \cite{wei2013load} necessary and sufficient conditions are provided to guarantee load balancing in the presence of input constraints but with no optimality.
It is shown that if the graph has uni-directional flow due to the saturation, a sufficient condition for output agreement is that the associated directed graph is strongly connected.
The same authors recently provided a result in which proportional-integral (PI) controllers are able to handle state constraints \cite{wei2014constrained}.

Inspired by \cite{wei2013load} and \cite{burger2015dynamic}, we consider a flow network with constant disturbances and saturated transfer rates on the links. Furthermore, we consider inputs on the nodes, which may also be subject to saturation. This is motivated by networks in which the inputs represent production rates, which have a minimal and/or maximal capacities. The main contributions of this paper are twofold. First, we provide two distributed controllers, one that regulates the input on each node and one that controls the flows on the edges. Building upon \cite{trip_2016_automatica}, \cite{Buerger2015}  and \cite{burger2015dynamic}, we show that these distributed controllers guarantee convergence to a quasi-optimal steady state, that is, to a steady state that is arbitrarily close to the optimal one. Second, we extend this result in the presence of heterogenous saturation on both inputs. In particular, we can enforce positivity constraints on the link flows, {\it i.e. }a network with unidirectional flows. In both cases we provide sufficient conditions for global asymptotic stability based on Lyapunov arguments. Finally, we apply these results to a district heating system with storage devices.

The structure of the paper is as follows. In Section \ref{setup} we introduce the model along with two problem formulations. The first problem considers optimal steady state inputs, whereas the second one is an extension, in which we additionally consider saturation on the inputs and the flows. The solution to the first problem is given in Section \ref{unconstraintcase} and the one to the second problem is given in Section \ref{constraintcase}. Finally, we present a case study in Section \ref{example}, followed by the conclusions in Section \ref{conclusion}.

\subsection{Notation}

Let $\mathbb{R}$ denote the set of real numbers and let $\mathbb{R}_{\geq0}$ be the set of non-negative real numbers. Similar to \cite{bapat2010graphs}, we define a directed graph $\mathcal{G'}$ as $\mathcal{G'}=(\mathcal{E'},\mathcal{V})$, where $\mathcal{V}$ is the set of vertices and $\mathcal{E'}$ is the set of directed edges. Furthermore, we define the undirected graph $\mathcal{G}$ as $\mathcal{G}=(\mathcal{E},\mathcal{V})$ where $\mathcal{E}$ contains the same, but undirected, vertices as in $\mathcal{E'}$. Corresponding to the direction of a directed edge, we assign a $-$ and $+$ at the ends, where it connects to a vertex, while for an undirected graph the $-$ and $+$ are assigned arbitrarily. Using this we introduce the incidence matrix $B \in \mathbb{R}^{n\times m}$, whose elements are defined as
$$b_{ik}=\left\{
     \begin{array}{ll}
       1 & :  \begin{array}{l} \text{if the $i$th node connects to}\\
       \text{the positive ($+$) end of edge $k$} \end{array}\\
       -1 & : \begin{array}{l} \text{if the $i$th node connects to}\\
       \text{the negative ($-$) end of edge $k$} \end{array}\\
       0 & : \begin{array}{l}\text{otherwise.} \end{array}
     \end{array}
   \right.$$
The Laplacian matrix is defined as $L=B B^T $ and let $\mathds{1}$ be the all ones vector. For any matrix $A$ we define $\operatorname{Im}(A)$ to be the image, $\text{ker}(A)$ to be the kernel and $A^\dagger$ to be the Moore-Penrose pseudo-inverse of $A$. For a vector space $\mathcal{S}$ we define $\mathcal{S}^\perp$ to be the orthogonal complement of $\mathcal{S}$, and let $\text{span}(\mathcal{S}) =  \left \{ {\sum_{i=1}^k \lambda_i x_i \Big| k \in \mathbb{N}, x_i  \in \mathcal{S}, \lambda _i  \in \mathbb{R}} \right \}$. For a vector $x\in\mathbb{R}^n$ we define $\|x\|$ to be a norm and the matrix norm is defined as
\be
\|A\| = \sup\{\|Ax\| : x\in \mathbb{R}^n \mbox{ with }\|x\|= 1\}. \nonumber
\ee
The $i$-th element of a vector $x$ is denoted as $(x)_i\in\mathbb{R}$, where the brackets are omitted if it causes no ambiguity.
Next we define the multidimensional saturation function $\text{sat}(x;x^-,x^+):\mathbb{R}^n\rightarrow \mathbb{R}^n$ as
\be
\text{sat}(x;x^-,x^+)_i:=\begin{cases} x_i^- &: \mbox{if } x_i \leq x_i^- \\
x_i &: \mbox{if } x_i^- < x_i < x_i^+ \\
x_i^+ &: \mbox{if }  x_i^+ \leq x_i,  \end{cases} \nonumber
\ee
where $x_i^-,x_i^+\in\mathbb{R}$. Lastly, for $a,b\in\mathbb{R}^n$ we define the inequalities ({\it e.g.} $a\leq b$) element-wise.

\section{FLOW NETWORKS}
\label{setup}

\subsection{Model}
We consider a network of physically linked undamped dynamical systems which can be represented by a graph $\mathcal{G}=(\mathcal{E},\mathcal{V})$, where $|\mathcal{E}|=m$ and $|\mathcal{V}|=n$. Each node $i$ has an input $(u_p)_i$ and a disturbance $d_i$, along with a state variable $x_i$. A second input $(u_e)_j$ is associated to each link $j$, which represents the transportation between the nodes. The dynamic model is as follows:
\begin{equation}
\begin{aligned}
\dot{x}(t)&=Bu_e(t)+u_p(t)+d\\
y(t)&=x(t)-\bar{x}(t),
\end{aligned}\label{model}
\end{equation}
where $x(t), u_p(t), y(t), d \in\mathbb{R}^n$ and $u_e(t)\in\mathbb{R}^m$. The inputs $u_p(t)$ and $u_e(t)$ are considered to be controllable and the disturbance $d$ is regarded as an unknown constant. Finally, we regard $\bar{x}(t)\in\mathbb{R}^n$ as the reference signal and we assume it to be of the form
\be
\bar{x}(t)=\bar{x}_0+\bar{x}_s t, \label{xbart}
\ee
where $\bar{x}_0\in\mathbb{R}^n$ and $\bar{x}_s\in\mathbb{R}^n$ are considered to be known constants.
\begin{rema}
The motivation to have a ramp for the reference signal $\bar{x}(t)$ comes from flow networks in which $x$ is considered to be a stored quantity. Namely, in these networks it can be desirable to have intervals during which constant charging ($\bar{x}_s>0$) or discharging ($\bar{x}_s<0$) occurs. To this end we will refer to $\bar{x}_s$ as the storage rate. Note that (\ref{xbart}) reduces to a standard constant setpoint in the case $\bar{x}_s=0$.
\end{rema}
To keep the notation as light as possible, we omit in the remainder of this paper the explicit dependence on $t$ of all the previously defined variables whenever is causes no confusion.

\subsection{Optimal feedforward input}
\label{optimalproduction}
In order to state our control problem we first define, similar to \cite{trip_2016_automatica}, \cite{Buerger2015} and \cite{burger2015dynamic}, an optimization problem whose optimum should be achieved at steady state. This is motivated by hydraulic and district heating networks (see {\it e.g. } \cite{DePersis2011}, \cite{DePersis2014} and \cite{Scholten2015}) that have producers on the nodes with heterogenous production costs. To this end, we assign a linear-quadratic input-dependent cost at each node, which is given by
\be
C_i((u_p)_i)=s_i+ r_i(u_p)_i+\frac{1}{2}q_i (u_p)_i^2, \label{individualcosts}
\ee
with $s_i, r_i \in\mathbb{R}$ and we assume that $q_i\in\mathbb{R}_{>0}$. Note that this assumption implies that (\ref{individualcosts}) is strictly convex. The total cost function we consider is given by $C(u_p)=\Sigma_{i=1}^n C_i((u_p)_i)$ which can be written as
\be
C(u_p)=s+r^T u_p+ \frac{1}{2}u_p^T Q u_p,
\ee
where $s:=\Sigma_{i=1}^n s_i$, $r=\left(
                              \begin{array}{cccc}
                                r_1 & \dots & r_n \\
                              \end{array}
                            \right)^T$ and $Q:=\text{diag}(q_1, \dots, q_n )$. Furthermore we want that the total total input matches the disturbance plus the prescribed storage rate $\bar{x}_s$ at steady state, {\it i.e.} $\mathds{1}^T(u_p+d-\bar{x}_s)=0$. For these reasons we consider the following optimization problem:
\begin{equation}
\begin{aligned}
& \underset{u_p}{\text{minimize}}
& & C(u_p) \\
& \text{subject to}
& & \mathds{1}^T(u_p+d-\bar{x}_s)=0.
\end{aligned}
\label{optimalprod}
\end{equation}
\begin{lemm}
The solution to (\ref{optimalprod}) is given by
\be
\overline{u}_p  = -Q^{-1} \left(\frac{\mathds{1}\mathds{1}^T }{\mathds{1}^T Q^{-1}\mathds{1}}(d-\bar{x}_s-Q^{-1}r )+r  \right).\label{qpopt}
\ee
\end{lemm}
\begin{proof}
The proof is standard and therefore omitted.
\end{proof}

We point out that $\overline{u}_p $ depends on the unmeasured disturbance $d$. Keeping this in mind and having obtained the expression (\ref{qpopt}), we are ready to define our control problems.
\subsection{Control problems}
\label{controlproblem}
We define two state regulation problems fulfilling optimality condition (\ref{qpopt}) at steady state.
\begin{problem}
\label{problem1}
Design distributed controllers that regulate the flow on the edges $u_e$ and input $u_p$ at the nodes such that
\begin{align}
\lim_{t \rightarrow \infty} \left|\left|x(t)-\bar{x}(t)\right|\right|&=0 \label{xgoal}\\
\lim_{t \rightarrow \infty} \left|\left|u_p(t)-\overline{u}_p \right|\right|&=0, \label{u_p_goal}
\end{align}
where $\overline{u}_p$ is as in (\ref{qpopt}) and $\bar{x}(t)$ is as in (\ref{xbart}).
\end{problem}
We extend this problem statement by considering constraints on the input. Furthermore, motivated by physical limitations, we impose uni-directional and maximal flow constraints on the edges. Hence, Problem \ref{problem2} is formulated as follows:
\begin{problem}
\label{problem2}
For any given positive (arbitrarily small) numbers $\epsilon_1$ and $\epsilon_2$, design distributed controllers that regulate the flows on the edges $u_e$ and input $u_p$ at the nodes such that
\begin{align}
\lim_{t \rightarrow \infty} \|x(t)-\bar{x}(t)\|&<\epsilon_1 \label{xgoal2}\\
\lim_{t \rightarrow \infty} \|u_p(t)-\overline{u}_p \|&<\epsilon_2, \label{u_p_goal2}
\end{align}
 where $\overline{u}_p$ is as in (\ref{qpopt}) and $\bar{x}\in\mathbb{R}^n$ is as in (\ref{xbart}). Furthermore,
\begin{subequations}\label{saturationbounds}
\begin{align}
u_p^-\leq& u_p(t) \leq u_p^+ \label{upsatbound}\\
0\leq& u_e(t) \leq u_e^+, \label{uesatbound}
\end{align}
\end{subequations}
should hold for all $t\geq0$.
\end{problem}
\begin{rema}
In contrast to asymptotical convergence as is considered in Problem \ref{problem1}, we resort to practical convergence in order to guarantee (\ref{upsatbound}) and (\ref{uesatbound}).
\end{rema}
We will now provide a solution to Problem \ref{problem1} in Section \ref{unconstraintcase} followed by a solution to Problem \ref{problem2} in Section \ref{constraintcase}.

\section{UNCONSTRAINED CASE}
\label{unconstraintcase}

 In this section we provide a solution to Problem \ref{problem1}, which also sets the ground for the controller design and analysis that solves Problem \ref{problem2}.

\subsection{Controller design}
\label{contwithoutsat}
To solve Problem \ref{problem1} we propose two controllers, one generating $u_e$ and one providing $u_p$. The former controller takes the outputs of the incident nodes as its input and takes the form of a standard PI controller. This controller is given by
\begin{equation}
\begin{aligned}
\dot{x}_e&=\gamma_e B^Ty\\
u_e&=- \gamma_c B^Ty-\gamma_e x_e,
\end{aligned}\label{conte}
\end{equation}
where $\gamma_e,\gamma_c\in \mathbb{R}_{>0}$ are suitable gains. 
The latter controller, takes its local error measurement $y$ as an input. To guarantee an optimal input at steady state we assign a state variable $({x}_p)_i$ to each node. This state is communicated via a connected communication network that is represented by\footnote{Note that the graph represented by $L_c$ does not necessarily have to coincide with the graph represented by $L$. } $L_c$. The underlying graph of this communication networks can be directed or undirect, where we assume in the latter case that it is strongly connected. This results in the following controller
\begin{subequations}\label{contp}
\begin{align}
\dot{x}_p&=-\gamma_l L_cx_p- \gamma_pQ^{-1}y\\
u_p&=Q^{-1}(\gamma_p{x}_p-r), \label{contp_up}
\end{align}
\end{subequations}
where $\gamma_l,\gamma_p \in \mathbb{R}_{>0}$ are suitable gains. The controller is fully distributed due to the diagonal form of $Q^{-1}$ and diffusive coupling between the states $x_p$. This coupling is required in order to achieve consensus of $x_p$ and we will prove that this implies that $u_p$ converges to the optimal steady state (\ref{qpopt}) despite the presence of disturbances.

Before we state the main theorem of this section we introduce the following lemma:

\begin{lemm}
Let
\be
  \bar{x}_p=-\frac{1}{\gamma_p}\frac{\mathds{1}\mathds{1}^T}{\mathds{1}^TQ^{-1}\mathds{1}}(d-\bar{x}_s-Q^{-1}r)\label{barx_def}
  \ee
  and $\bar{x}_e$ be any solution to
  \be
  \gamma_eB\bar{x}_e=\left(I-\frac{Q^{-1}\mathds{1}\mathds{1}^T}{\mathds{1}^TQ^{-1}\mathds{1}}\right)\left(d-\bar{x}_s-Q^{-1}r\right), \label{Bbarxe}
 \ee
 then the incremental states
\begin{equation}\label{coordinatechange1}
\begin{aligned}
\tilde{x}=& \medspace  x-\bar{x}\\
\tilde{x}_p=& \medspace x_p-\bar{x}_p\\
\tilde{x}_e=& \medspace x_e-\bar{x}_e,
\end{aligned}
\end{equation}
with $x$, $x_p$ and $x_e$ as a solution to system (\ref{model}), in closed loop with controllers (\ref{conte}) and (\ref{contp}), satisfy
 \begin{equation}
\begin{aligned}
\dot{\tilde{x}}&=-\gamma_cB B^T\tilde{x}-\gamma_eB\tilde{x}_e+\gamma_pQ^{-1}\tilde{x}_p\\
\dot{\tilde{x}}_p&=-\gamma_lL_c{\tilde{x}}_p -\gamma_pQ^{-1}\tilde{x}\\
\dot{\tilde{x}}_e&=\gamma_eB^T\tilde{x}.
\end{aligned}\label{closedloopsys1}
\end{equation}
Furthermore, a solution to (\ref{Bbarxe}) always exists.
\end{lemm}

\begin{proof}
We combine (\ref{model}) with (\ref{conte}) and (\ref{contp}), to obtain the closed loop system
\begin{equation}
\begin{aligned}
\dot{x}=&-\gamma_cB B^T(x-\bar{x})-\gamma_eBx_e \\
&+Q^{-1}(\gamma_p{x}_p-r)+d\\
\dot{x}_p=&-\gamma_lL_c{x}_p -\gamma_pQ^{-1}(x-\bar{x})\\
\dot{x}_e=& \medspace \gamma_eB^T(x-\bar{x}).
\end{aligned}\label{nosat_eq}
\end{equation}
In light of (\ref{coordinatechange1}) it follows directly from (\ref{nosat_eq}) that (\ref{closedloopsys1}) is satisfied. Lastly we prove that there exists a $\bar{x}_e$ that satisfies (\ref{Bbarxe}). Since
$\operatorname{Im}(B)=\text{ker}(B^T)^\perp=\text{span}(\mathds{1})^\perp$ and $\mathds{1}^T\left(I-\frac{Q^{-1}\mathds{1}\mathds{1}^T}{\mathds{1}^TQ^{-1}\mathds{1}}\right)=0$, we know that there always exists a $\bar{x}_e$ that satisfies (\ref{Bbarxe}), which concludes the proof.
\end{proof}

\begin{rema}
The closed loop dynamics (\ref{closedloopsys1}) are similar to the linear version of the closed loop dynamics of power grids, as studied in \cite{trip_2016_automatica}. The main difference in the model considered here, which requires a modification of the analysis, is the lack of damping terms.
\end{rema}

We will now state the following theorem which gives sufficient conditions to solve Problem \ref{problem1}.
\begin{theo}
If the graph $\mathcal{G}$ is connected and there exists a pair of entries $q_i,q_j$ such that $q_i\neq q_j$, then controllers (\ref{conte}) and (\ref{contp}), in closed loop with (\ref{model}), solve Problem \ref{problem1}. \label{theorem1}
\end{theo}
\begin{proof}
In order to analyse the stability of the system, we use a standard quadratic Lyapunov function
\be
V(\tilde{x},\tilde{x}_e,\tilde{x}_p)=\frac{1}{2}\|\tilde{x}\|^2+\frac{1}{2}\|\tilde{x}_p\|^2 +\frac{1}{2}\|\tilde{x}_e\|^2.
\ee
Using (\ref{closedloopsys1}) it is easy to see that its derivative is given by
\begin{equation}
\dot{V}(\tilde{x},\tilde{x}_e,\tilde{x}_p)=-\gamma_c\|B^T\tilde{x}\|^2-\gamma_l\|B_c^T{\tilde{x}}_p\|^2, \label{vdotresult}
\end{equation}
where $B_c$ is the incidence matrix of the communication graph and satisfies $B_cB_c^T=L_c$. Due to the quadratic form of $V$, it is clear that $V$ is positive definite and radially unbounded. Using LaSalle's invariance principle we can conclude that $(\tilde{x},\tilde{x}_e,\tilde{x}_p)$ converges to the largest invariant set where $\dot{V}(\tilde{x},\tilde{x}_e,\tilde{x}_p)=0$, which is given by
\be
\mathcal{S}:=\left\{(\tilde{x},\tilde{x}_e,\tilde{x}_p)|B^T\tilde{x}=0,  B_c^T{\tilde{x}}_p=0 \right\}. \label{invariantset}
\ee
Next we characterize the dynamics on this invariant set $\mathcal{S}$. These, in light of (\ref{closedloopsys1}), are given by
\begin{subequations}\label{lasalle_tot}
\begin{align}
\dot{\tilde{x}}&=-\gamma_eB\tilde{x}_e+\gamma_pQ^{-1}{\tilde{x}}_p\label{lasalle1}\\
\dot{\tilde{x}}_p&=-\gamma_pQ^{-1}\tilde{x}\label{lasalle2}\\
\dot{\tilde{x}}_e&=0.\label{xtildeezero}
\end{align} \label{closedlooponS}
\end{subequations}
Since the graphs that represent the physical interconnections and communications are both connected, we see that on  $\mathcal{S}$ both $\tilde{x}=\mathds{1}\tilde{x}^*$ and ${\tilde{x}}_p=\mathds{1}{\tilde{x}}_p^*$ are satisfied, where $\tilde{x}^*$ and $\tilde{x}_p^*$ are undetermined scalar functions.  Together with (\ref{lasalle2})
 we conclude that
\be
\gamma_p(q^{-1}_j -q^{-1}_i) {\tilde{x}}^*= 0 \quad \text{for all } i,j. \label{xdotandqiszero}
\ee
By assumption there exist an $i$ and $j$ such that $q^{-1}_j \neq q^{-1}_i$, which implies, together with (\ref{xdotandqiszero}), that $\tilde{x}^*= 0$ and therefore also that $(\dot{\tilde{x}}_p)^*= 0$. By evaluating the dynamics of (\ref{closedlooponS}) we obtain
 \begin{align}
\tilde{x}&=0\label{xtildeiszero}\\
\tilde{x}_p&=\frac{\gamma_e}{\gamma_p}QB\tilde{x}_e, \label{xtildepisqbxe}
\end{align}
from which it follows that $\mathds{1}^TQ^{-1}\tilde{x}_p=0$. Since we also know from (\ref{invariantset}) that ${\tilde{x}}_p=\mathds{1}({\tilde{x}}_p)_i$ this implies that $\tilde{x}_p=0$. From this and (\ref{contp_up}) we can now conclude that
 \begin{equation}
\bar{u}_p = -Q^{-1}\left( \frac{\mathds{1}\mathds{1}^T}{\mathds{1}^TQ^{-1}\mathds{1}}(d-\bar{x}_s-Q^{-1}r) +r\right),
\label{upisoptimal}
 \end{equation}
 which coincides with the optimal steady state input in view of (\ref{qpopt}). By (\ref{xtildeiszero}) and (\ref{upisoptimal}) we conclude that Problem \ref{problem1} is solved.
\end{proof}

\begin{rema}
By taking the gains $\gamma_p$ non-identical at different nodes, the condition $q_i\neq q_j$ in Theorem \ref{theorem1} can be relaxed to $(\gamma_p)_jq^{-1}_j \neq (\gamma_p)_iq^{-1}_i$. This implies that if $q_j=q_i$ for all $i\neq j$, applying heterogenous gains would still guarantee convergence.
\end{rema}

\section{CONSTRAINED CASE}
\label{constraintcase}
In this section we provide a solution to Problem \ref{problem2} where, compared to Problem \ref{problem1}, we additionally have constraints (\ref{saturationbounds}) on the inputs $u_p$ and $u_e$. We propose controllers that are similar to those presented in Section \ref{unconstraintcase}, while taking these additional constraints into account. To this end we modify (\ref{conte}) in order to satisfy (\ref{uesatbound}) and propose the following controller to generate $u_e$:
\begin{subequations} \begin{align}
\dot{x}_e&=\gamma_e B^Ty\\
u_e&=\text{sat}(-\gamma_cB^Ty- \gamma_ex_e;0,u_e^+),\label{contesatb}
\end{align}\label{contesat}%
\end{subequations}
where $\gamma_e, \gamma_c \in \mathbb{R}$ are appropriate gains. Note that the network has uni-directional flows since the lower bound of the saturation is identical to zero. For this reason the graph $\mathcal{G}'$, that models the physical interconnection, can be viewed as a directed one. We let the directions of the edges in $\mathcal{E}'$ be such that they coincide with the permitted flow directions.

The controller that regulates the input on the nodes $u_p$ uses the same principles as (\ref{contp}), with some additions in order to satisfy (\ref{upsatbound}). To this end, we saturate the output of this controller. However, this is not sufficient to guarantee convergence. For this reason we adjust the dynamics of the controller to
\begin{subequations}\label{Upsatcontroller}
\begin{align}
\dot{x}_p=&-\gamma_l L_c\text{sat}( x_p ; \frac{1}{\gamma_p}(Q u_p^- +r)  , \frac{1}{\gamma_p}(Q  u_p^+  +r))\nonumber\\
& -\gamma_pQ^{-1}\left(y -\gamma_cBu_e \right) \label{dotxpsat}\\
u_p=&\text{sat}(Q^{-1}(\gamma_p{x}_p-r); u_p^-  ,u_p^+),\label{upsat}
\end{align}
\end{subequations}
where $L_c$ is the Laplacian of a connected communication graph, $\gamma_c$ is as in (\ref{contesat}) and $\gamma_l, \gamma_p \in \mathbb{R}$ are appropriate gains.

\begin{rema}
Note that (\ref{dotxpsat}) has $\gamma_p\gamma_cQ^{-1}Bu_e$ as an additional term compared to (\ref{contp}). We will show that this term play a key role to prove convergence but it also causes a steady state error. Interestingly, this error can be made arbitrarily small by adjusting the gains $\gamma_c$, $\gamma_p$ and $\gamma_l$. The practical consequence of this term is that the controller additionally needs to measure the difference between all the incoming and outgoing flows. Since these measurements are available locally, controller (\ref{Upsatcontroller}) is still fully distributed.
\end{rema}

\begin{rema}
We observe that (\ref{model}) in closed loop with controllers (\ref{contesat}) and (\ref{Upsatcontroller}) is globally Lipschitz. For this reason we can conclude that a solution exists for all time $t\geq0$.
\end{rema}

Before we state our main theorem we define a change of coordinates in which we distinguish the desired steady state and the steady state deviation from the desired one, which we denote with a bar and hat, respectively. To this end, we let
\be
\begin{aligned}
\tilde{x}=& \medspace x-\bar{x}-\hat{x}&\\
\tilde{x}_e=& \medspace x_e-\bar{x}_e-\hat{x}_e\\
\tilde{x}_p=& \medspace x_p-\bar{x}_p-\hat{x}_p,
\end{aligned}\label{coordinatechange}
\ee
where $\bar{x}_p$ is as in (\ref{barx_def}), $\bar{x}_e$ is any solution to (\ref{Bbarxe})
and we define $\hat{x}$, $\hat{x}_e$ and $\hat{x}_p$ as the solution to
 \begin{subequations}
 \begin{align}
 0 =& B^T\hat{x}   \label{xerr}\\
0 =& -\gamma_e B \hat{x}_e + \gamma_p Q^{-1}\hat{x}_p \label{xeerr}\\
0 =& -\gamma_pQ^{-1}\hat{x}-\gamma_l L_c \hat{x}_p \nonumber\\
& - \gamma_p\gamma_c\gamma_eQ^{-1}B(\hat{x}_e +\bar{x}_e). \label{xperr}
 \end{align} \label{hatsss}
\end{subequations}
that has minimal Euclidean norm.

The next lemmas show that the solutions to (\ref{hatsss}) always exist and derive an incremental form for system (\ref{model}) in closed loop with  (\ref{contesat}) and (\ref{Upsatcontroller}) in suitable new coordinates.

\begin{lemm}
Solutions $\hat{x}$, $\hat{x}_e$ and $\hat{x}_p$ to (\ref{hatsss}) always exist and are given by:
\begin{align}
\hat{x}_p=&\frac{\gamma}{\gamma_p}Q\left(\gamma\bar{Q}+\Phi\right)^{-1}\bar{Q}^2 Q\tilde{d}\label{hatxp}\\
\hat{x}_e=&\frac{\gamma}{\gamma_e}B^\dagger \left(\gamma\bar{Q}+\Phi\right)^{-1}\bar{Q}^2 Q\tilde{d} \label{hatxe}\\
\hat{x}=& \gamma_c\frac{\mathds{1}\mathds{1}^TQ^{-1}}{\mathds{1}^TQ^{-1}\mathds{1}}\left(I- \gamma\left(\gamma\bar{Q}+\Phi\right)^{-1}\bar{Q} \right)\bar{Q} Q\tilde{d},\label{hatx}
\end{align}
and
\begin{subequations}
\begin{align}
\gamma:=&\gamma_p^2\frac{\gamma_c}{\gamma_l}\\
\tilde{d}:=&d-\bar{x}_s-Q^{-1}r\label{tilded}\\
\bar{Q}:=& \frac{Q^{-1}\mathds{1}\mathds{1}^TQ^{-1}}{\mathds{1}^TQ^{-1}\mathds{1}}-Q^{-1} \label{Qbar}\\
\Phi:=&-L_cQ+\frac{1}{n}\mathds{1}\mathds{1}^T. \label{Phi}
\end{align}
\end{subequations}
\label{lemma_bound_gamma}
\end{lemm}

\begin{proof}
%
From (\ref{xerr}) and (\ref{xeerr})  we obtain that
\begin{align}
 \hat{x}&= \mathds{1}\hat{x}^* \label{xstar}\\
 0&=\mathds{1}^TQ^{-1}\hat{x}_p, \label{sumQinfxhatp}
 \end{align}
 for some scalar function $\hat{x}^*$. From  (\ref{Bbarxe}) we can see that $\bar{x}_e$ is any solution to
 \begin{align}
\gamma_eB\bar{x}_e=&-\bar{Q}Q\tilde{d},  \label{xbare_sol}
 \end{align}
 which combined this with (\ref{xstar}), (\ref{xeerr}) and (\ref{xperr}) results in
\be
\begin{aligned}
&\left( L_c +\gamma_p^2\frac{\gamma_c}{\gamma_l}Q^{-2}\right)\hat{x}_p =\\
& -\frac{\gamma_p}{\gamma_l}Q^{-1}\mathds{1}\hat{x}^*  + \gamma_p\frac{\gamma_c}{\gamma_l} Q^{-1}\bar{Q}Q\tilde{d}. \label{Lc_temp}
\end{aligned}
\ee
By solving for $\hat{x}^*$ we obtain
  \be
\hat{x}^*  =  -\gamma_c\frac{\mathds{1}^TQ^{-1}}{\mathds{1}^TQ^{-1}\mathds{1}}\left(\gamma_pQ^{-1}\hat{x}_p-\bar{Q}Q\tilde{d} \right).\label{hatxstarsolved}
 \ee
Substituting (\ref{hatxstarsolved}) in (\ref{Lc_temp}) and combining this with (\ref{sumQinfxhatp}) yields
     \begin{align}
  \tilde{Q}Q^{-1}\hat{x}_p=& \left(
    \begin{array}{c}
         \gamma_p\frac{\gamma_c}{\gamma_l}\bar{Q} \\
         0 \\
    \end{array}
  \right)\bar{Q}Q\tilde{d}.\label{xhatpnotsolved}
  \end{align}
  where $\bar{Q}$ is as in (\ref{Qbar}) and $\tilde{Q}$ is defined as
  \be
   \tilde{Q} := \left(
    \begin{array}{c}
      \gamma \bar{Q}-L_cQ \\
      \mathds{1}^T \\
    \end{array}
  \right).\label{Qtilde}
   \ee
Due to Lemma \ref{lemmaAcolumns} in Appendix \ref{appendix} we know that $(\tilde{Q}^T\tilde{Q})^{-1}$, $(\gamma \bar{Q}-L_cQ+\mathds{1}\mathds{1}^T)^{-T}$ and $(\gamma \bar{Q}-L_cQ+\frac{1}{n}\mathds{1}\mathds{1}^T)^{-1}$ exists. This implies that the solution of (\ref{xhatpnotsolved}) is given by (\ref{hatxp}), since
\begin{align}
  \hat{x}_p=& \frac{1}{\gamma_p}Q(\tilde{Q}^T\tilde{Q})^{-1}\tilde{Q}^T \left(
                                              \begin{array}{c}
                                                \gamma\bar{Q} \\
                                                0 \\
                                              \end{array}
                                            \right)
   \bar{Q}Q\tilde{d} \nonumber \\
      =&\frac{\gamma}{\gamma_p}Q\left((\gamma \bar{Q}-L_cQ+\mathds{1}\mathds{1}^T)^T(\gamma \bar{Q}-L_cQ+\frac{\mathds{1}\mathds{1}^T}{n})\right)^{-1} \nonumber \\
   &(\gamma \bar{Q}-L_cQ+\mathds{1}\mathds{1}^T)^T \bar{Q}^2Q\tilde{d} \nonumber \\
   =&\frac{\gamma}{\gamma_p}Q\left(\gamma \bar{Q}-L_cQ+\frac{1}{n}\mathds{1}\mathds{1}^T\right)^{-1}\bar{Q}^2 Q\tilde{d}, \label{xphatsol}
\end{align}

where we used the identities $\mathds{1}^T\bar{Q}=0$ and (\ref{fullrankmatrices}) in Appendix \ref{appendix}. By combining (\ref{xstar}), (\ref{hatxstarsolved}) and (\ref{xphatsol}) we immediately observe that (\ref{hatx}) is satisfied.

To find $\hat{x}_e$ we use (\ref{xeerr}) and obtain
\be
 B\hat{x}_e  =   \frac{\gamma_p}{\gamma_e} Q^{-1}\hat{x}_p. \label{Bhatxe}
\ee
To prove that (\ref{Bhatxe}) has a solution we use the identity $\mathds{1}^T(\gamma\bar{Q}+\Phi)=\mathds{1}^T$ and from Lemma \ref{lemmaAcolumns} in Appendix \ref{appendix} we know that $(\gamma\bar{Q}+\Phi)$ is invertible. This implies that $\mathds{1}^T ( \gamma\bar{Q}+\Phi  )^{-1} = \mathds{1}^T$ and therefore we have that $\mathds{1}^T \left(\gamma\bar{Q}+\Phi\right)^{-1}\bar{Q}=0$. Since $\operatorname{Im}(B)=\text{ker}(B^T)^\perp=\text{span}(\mathds{1})^\perp$ we conclude that (\ref{Bhatxe}) has a solution. Moreover, a solution with the minimal Euclidean norm is given by
\be
\hat{x}_e  =  B^\dagger \frac{\gamma_p}{\gamma_e} Q^{-1}\hat{x}_p,  \label{infsolforhatxe}
\ee
and due to (\ref{xphatsol}) it is easy to see that (\ref{infsolforhatxe}) coincides with (\ref{hatxe}). Lastly it can be checked that (\ref{hatxp})-(\ref{hatx}) satisfies (\ref{hatsss}) identically, which concludes this proof.
%
%
%
\end{proof}


\begin{lemm}
The incremental states as in (\ref{coordinatechange}), where $x$, $x_e$ and $x_p$ are the solution to (\ref{model}) in closed loop with  (\ref{contesat}) and (\ref{Upsatcontroller}), and $\bar{x}_e$ as any solution to (\ref{Bbarxe}), $\bar{x}_p$ as in (\ref{barx_def}) and $\hat{x}_p$, $\hat{x}_e$, $\hat{x}$ as defined as in (\ref{hatsss}), satisfy
\begin{equation}
\begin{aligned}
\dot{\tilde{x}}=& B \emph{sat}_e(\tilde{x},\tilde{x}_e)+\gamma_p Q^{-1}\emph{sat}_p(\tilde{x}_p)\\
\dot{\tilde{x}}_p=&-\gamma_lL_c\emph{sat}_p(\tilde{x}_p) +\gamma_p\gamma_cQ^{-1}B\emph{sat}_e(\tilde{x},\tilde{x}_e)\\
&-\gamma_pQ^{-1}\tilde{x}\\
\dot{\tilde{x}}_e=&\gamma_e B^T\tilde{x},
\end{aligned}\label{sat_dynamics}
\end{equation}
where
 \begin{align}
\emph{sat}_e(\tilde{x},\tilde{x}_e):=& \medspace \emph{sat}(-\gamma_cB^T\tilde{x}-\gamma_e\tilde{x}_e;x_p^-, x_p^+)\label{sate} \\
\emph{sat}_p(\tilde{x}_p):=& \medspace \emph{sat}({\tilde{x}}_p;x_p^-,x_p^+),\label{satp}
\end{align}
with
 \begin{align}
x_e^-&=\gamma_e(\bar{x}_e+\hat{x}_e) \label{xemindef} \\
x_e^+&=\gamma_e(\bar{x}_e+\hat{x}_e)+u_e^+ \label{xemaxdef}\\
x_p^-&=\frac{1}{\gamma_p}(Qu_p^-+r)-(\bar{x}_p+\hat{x}_p) \label{xpmindef}\\
x_p^+&=\frac{1}{\gamma_p}(Qu_p^++r)-(\bar{x}_p+\hat{x}_p). \label{xpmaxdef}
\end{align}
\label{lemma_closedloop}
\end{lemm}

\begin{proof}
We first write system (\ref{model}) in closed loop with (\ref{contesat}) and (\ref{Upsatcontroller}) and obtain
\begin{equation}
\begin{aligned}
\dot{{x}}=&\gamma_pQ^{-1}\text{sat}({x}_p; \frac{1}{\gamma_p}(Qu_p^-+r), \frac{1}{\gamma_p}(Qu_p^++r)) \\
&+\medspace B \text{sat}(-\gamma_cB^T(x-\bar{x})-\gamma_ex_e;0,u_e^+)+\bar{d}\\
&-Q^{-1}r\\
\dot{x}_p=&-\gamma_lL_c\text{sat}({x}_p;  \frac{1}{\gamma_p}(Q u_p^- +r)  , \frac{1}{\gamma_p}(Qu_p^+  +r))\\
& -\gamma_pQ^{-1}(x-\bar{x})+ \gamma_p\gamma_cQ^{-1}B\cdot\\
&\quad \quad \text{sat}(-\gamma_cB^T(x-\bar{x})-\gamma_ex_e;0,u_e^+) \\
\dot{x}_e=&\medspace \gamma_eB^T(x- \bar{x} ),
\end{aligned}\label{satfull}
\end{equation}
where we used the identities  $\text{sat}(A^{-1}x;x^-,x^+)=A^{-1}\text{sat}(x;Ax^-,Ax^+)$   and $\text{sat}(x+a;x^-,x^+)=\text{sat}(x;x^--a,Ax^+-a)+a$. Using (\ref{barx_def}), (\ref{Bbarxe}), (\ref{coordinatechange}) and (\ref{hatsss}) we can see that (\ref{satfull}) gives the desired result.
\end{proof}

Suppose that the steady states are unsaturated and $y=0$, then (\ref{contesatb}) and (\ref{upsat}) at steady state read as
\begin{align}
  \bar{u}_p =& \gamma_pQ^{-1}\bar{x}_p-r  \label{updef}\\
  \bar{u}_e =& -\gamma_e\bar{x}_e.  \label{upedef}
\end{align}
This implies, in view of (\ref{barx_def}) and (\ref{Bbarxe}) that
\begin{align}
  \bar{u}_p =& -\frac{Q^{-1}\mathds{1}\mathds{1}^T}{\mathds{1}^TQ^{-1}\mathds{1}}\left(d-\bar{x}_s-Q^{-1}r\right)-r \\
  B\bar{u}_e =& \bar{Q}Q\left(d-\bar{x}_s-Q^{-1}r\right),
\end{align}
where $\bar{Q}$ is as defined in (\ref{Qbar}). We note that (\ref{updef}) is, in light of (\ref{qpopt}), the desired steady state input. Furthermore it is important to note that $\bar{u}_p$ and $\bar{u}_e$ are independent of any gain parameters.

A sufficient condition to guarantee that the desired steady state exists is that the steady state inputs $\bar{u}_p$ and $\bar{u}_e$ are unsaturated. Furthermore, we will show that a sufficient condition to guarantee that this steady state is attractive is that the steady state inputs are strictly unsaturated. For these reasons we introduce the following definition.
\begin{definition}
  {\it(Feasibility condition)}.
  Given $d$, $r$ and $Q$, let $\bar{x}_p$ be as in (\ref{barx_def}) and let $\bar{x}_e$ be any solution of (\ref{Bbarxe}). We say that $u_e^+$, $u_p^-$ and $u_p^+$ satisfy the feasibility condition if there exist $\bar{u}_p$ and $\bar{u}_e$, as in (\ref{updef}) and (\ref{upedef}), such that
     \begin{align}
    u_p^-&<  \bar{u}_p < u_p^+, \label{ineqmachtingcondition}\\
    0&<  \bar{u}_e < u_e^+ \label{u_e^+min}.
    \end{align}
\label{matchingcondition}
\end{definition}

Before we state Theorem \ref{maintheorem} we define
\begin{align}
  \hat{u}_p :=& \gamma_pQ^{-1}\hat{x}_p  \label{uphat}\\
  \hat{u}_e :=& -\gamma_e\hat{x}_e,  \label{upehat}
\end{align}
with $\hat{x}_p$ and $\hat{x}_e$ as in (\ref{hatxp}) and (\ref{hatxe}), respectively. We refer to $\hat{u}_p$ and $\hat{u}_e$ as the steady state input errors. We are now ready to state the main result of this paper.

\begin{theo}
Let $u_e^+$, $u_p^-$ and $u_p^+$ satisfy the feasibility condition for a given $d$, $r$ and $Q$. Then Problem \ref{problem2} is solved by controllers (\ref{contesat})-(\ref{Upsatcontroller}) with a suitable choice of $\gamma_c$, $\gamma_p$ and $\gamma_l$ if:
\begin{enumerate}
\item there exists at least one pair $q_i,q_j$ such that $q_i\neq q_j$,
\item the directed graph $\mathcal{G}'$ is strongly connected,
\end{enumerate}
\label{maintheorem}
\end{theo}

\begin{proof}
In order to prove Theorem \ref{maintheorem} we will show that $\lim_{t\rightarrow\infty} \tilde{x}=0$ and $\lim_{t\rightarrow\infty} \tilde{x}_p=0$ and argue that this implies that Problem \ref{problem2} is solved. Let $V$ be as in Lemma \ref{lemma1} in Appendix \ref{appendix}. Using this same Lemma we know that we can invoke LaSalle's invariance principle to show that $(\tilde{x},\tilde{x}_e,\tilde{x}_p)$ converges to the largest invariant set where $\dot{V}=0$, which is given by
\be
\begin{aligned}
\mathcal{S}:=&\left\{(\tilde{x},\tilde{x}_e,\tilde{x}_p)|B\text{sat}_e(\tilde{x},\tilde{x}_e)=0,\right. \\
&\left.B_c^T\text{sat}_p(\tilde{x}_p)=0 \right\},
\end{aligned}\label{invariantset2}
\ee
with $\text{sat}_e(\tilde{x},\tilde{x}_e)$ as in (\ref{sate}) and $\text{sat}_p(\tilde{x}_p)$ as in (\ref{satp}). In light of (\ref{sat_dynamics}), we can see that the dynamics on this invariant set $\mathcal{S}$ are given by
\begin{subequations}\label{lasallesat_equations}
\begin{align}
\dot{\tilde{x}}&=\gamma_pQ^{-1}\text{sat}_p (\tilde{x}_p) \label{lasallesat1}\\
\dot{\tilde{x}}_p&=-\gamma_pQ^{-1}\tilde{x}\label{lasallesat2}\\
\dot{\tilde{x}}_e&=B^T\tilde{x}.\label{lasallesat3}
\end{align}
\end{subequations}
First we will prove that on this invariant set $\mathcal{S}$, necessarily $\tilde{x}_p=0$.

Let ${x}_p^-$ and ${x}_p^+$ be as in (\ref{xpmindef}) and (\ref{xpmaxdef}), respectively, then by Lemma \ref{bounds_not_sat} in Appendix \ref{appendix} we know that ${x}_p^-<0$ and ${x}_p^+>0$. Now assume by contradiction that there exists a $(\tilde{x}_p)_i$, which is not identically equal to zero. Now consider two cases, either $(\tilde{x}_p)_j=0$ for all $j\neq i$ or there exists at least one other $(\tilde{x}_p)_j$, with $i\neq j$, which is not identically equal to zero.
In the first case we have that
\begin{equation}
\text{sat}((\tilde{x}_p)_j;({x}_p^-)_j,({x}_p^+)_j)=0, \label{sat0}
\end{equation}
for each $j\neq i$, since $({x}_p^-)_j<0$, $({x}_p^+)_j>0$. Furthermore, since $B_c$ is the incidence matrix of a (strongly) connected graph, it holds that $B_c^T\text{sat}_p(\tilde{x}_p)=0$, which implies that
\be
\begin{aligned}
&\text{sat}(({\tilde{x}}_p)_i;({x}_p^-)_i,({x}_p^+)_i)\\
=&\text{sat}(({\tilde{x}}_p)_j;({x}_p^-)_j,({x}_p^+)_j),
\end{aligned} \label{sats_are_equal}
\ee
for each $i$ and $j$. From (\ref{sat0}) and (\ref{sats_are_equal}) we can now conclude that $({\tilde{x}}_p)_i=0$ since also $({x}_p^-)_i<0$, $({x}_p^+)_i)>0$. Therefore we have a contradiction and necessarily $(\tilde{x}_p)_i=0$.

%
%

Now consider the second case, where we assume that there exists at least another $(\tilde{x}_p)_j$, with $i\neq j$, which is not identically equal to zero. By (\ref{lasallesat1}) and (\ref{lasallesat2}) we obtain that $\ddot{\tilde{x}}_p=-\gamma_p^2Q^{-2}\text{sat}_p(\tilde{x}_p)$, which implies that for each element $i$ we have that
\be
(\ddot{\tilde{x}}_p)_i=\begin{cases}
-\gamma_p^2q_i^{-2}({x}_p^-)_i &\text{if } (\tilde{x}_p)_i\leq({x}_p^-)_i\\
-\gamma_p^2q_i^{-2}({x}_p^+)_i &\text{if }({x}_p^+)_i\leq(\tilde{x}_p)_i\\
-\gamma_p^2q_i^{-2}(\tilde{x}_p)_i &\text{otherwise}.
\end{cases}\label{dynamicsxptilde}
\ee
Let $p^-:=\max_i ({x}_p^-)_i $ and $p^+:=\min_i ({x}_p^+)_i$. Now we see that the solution $(\tilde{x}_p)_i$ to (\ref{dynamicsxptilde}) consists of parts that are periodic when the saturation is inactive, and are parabolic when the saturation is active. The intervals in which $(\tilde{x}_p)_i$ has a parabolic behaviour have a finite length, since $(\ddot{\tilde{x}}_p)_i<0$ if $(\tilde{x}_p)_i>0$ and $(\ddot{\tilde{x}}_p)_i>0$ if $(\tilde{x}_p)_i<0$, ensuring that it enters the unsaturated range. Furthermore, it is easy to see that in the unsaturated range the periodic behaviour forces $(\tilde{x}_p)_i$ to cross the origin in finite time. For this reason there exists an interval $(T_1,T_2)$ such that
\be
p^-\leq(\tilde{x}_p)_i\leq p^+,
\ee
on which, by definition of $p^-$ and $p^+$, all the saturations are inactive. This, together with (\ref{invariantset2}) implies that $\tilde{x}_p=\mathds{1}\alpha(t)$, where $\alpha(t)\in\mathbb{R}$. Due to (\ref{lasallesat1}) and (\ref{lasallesat2}), we have that
\be
\mathds{1}\ddot{\alpha}(t)=-Q^{-2}\mathds{1}{\alpha}(t),
\ee
which implies that
\be
-q_i^{-2}{\alpha}(t)=-q_j^{-2}{\alpha}(t),
\ee
for some $i,j$. Now, by assumption we know that there exists an $i$ and a $j$ such that $q_i\neq q_j$, which implies that $\alpha(t)=0$. It follows that $\tilde{x}_p(t)=0$ for $t\in(T_1,T_2)$. Also note that $(\tilde{x}_p)_i$ enters the interval $(T_1,T_2)$ in finite time and by (\ref{dynamicsxptilde}) we see that $(\tilde{x}_p)_i$ is locally Lipschitz continuous, hence cannot undergo jumps. This implies that $\tilde{x}_p(t)=0$ for all $t\geq0$ on the invariant set $\mathcal{S}$ which is a contradiction, implying that at most one $(\tilde{x}_p)_i$ is not identically equal to zero. As this case has already been ruled out, we obtain that $\tilde{x}_p(t)=0$ for all $t\geq0$ on the invariant set $\mathcal{S}$.

It is now trivial to prove that $\tilde{x}=0$ on $\mathcal{S}$. Due to (\ref{lasallesat2}) we can see that
\be
0=-Q^{-1}\tilde{x},
\ee
which implies that $\tilde{x}=0$.

Finally, due to a suitable choice of $\gamma_c$, $\gamma_p$ and $\gamma_l$, Lemma \ref{lemma3} in Appendix \ref{appendix} and (\ref{coordinatechange}), we have that
\be
\begin{aligned}
\lim_{t\rightarrow\infty}\| x-\bar{x}\|=&\lim_{t\rightarrow\infty}\|\tilde{x}+\hat{x}\|=\|\hat{x}\|<\epsilon_1\\
\lim_{t\rightarrow\infty}\| u_p-\bar{u}_p\|=&\lim_{t\rightarrow\infty}\|\tilde{u}_p+\hat{u}_p\|=\|\hat{u}_p\|<\epsilon_2,\\
\end{aligned}
\ee
and therefore the thesis follows.
\end{proof}


\begin{rema}
The choice of $\gamma_c$, $\gamma_p$ and $\gamma_l$ for which Problem \ref{problem2} is solved are explicitly constructed in Lemma \ref{lemma3} in Appendix \ref{appendix}. A discussion on this choice can be found in Remark \ref{remarkgains} in Appendix \ref{appendix}.
\end{rema}
%

\section{Case study}
\label{example}
\begin{figure}
\centering
  \includegraphics[trim=1cm 12cm 17cm 0.5cm, clip=true, width=0.37 \textwidth]{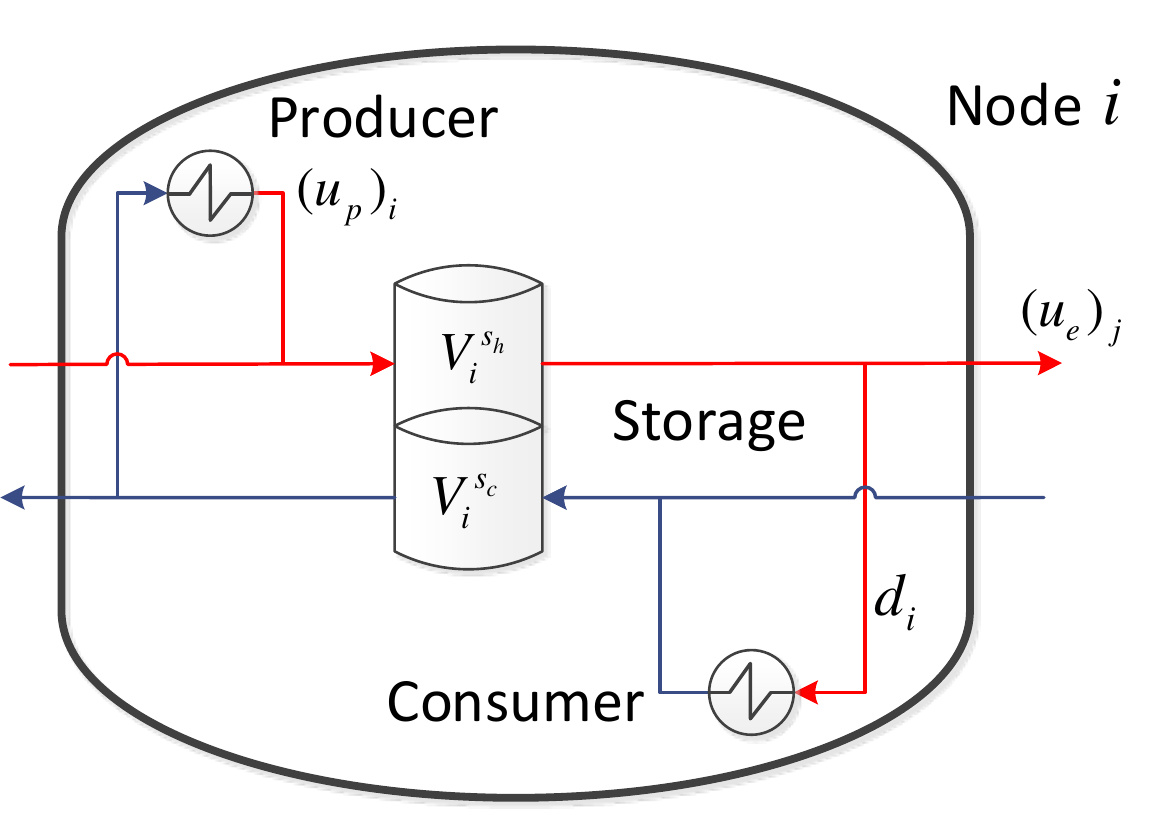}\\
  \caption{A node in the network}\label{node}
\end{figure}
Motivated by our previous work \cite{Scholten2015}, we provide a case study in which we consider a district heating system. The setup is such that each node has a producer, a consumer and a stratified storage tank. This storage tank has a hot and cold layer of water of which the variable volumes are denoted as $V_i^{S_h}$ and $V_i^{S_c}$, respectively and are both given in $m^3$. The topology of a node is given in Figure \ref{node}, and these nodes are connected via a graph $\mathcal{G}$. Using mass conservation laws, we obtain the dynamics for the hot and cold storage layers. These dynamics are given by
\begin{align}
\dot{V}^{S_h}&=B\theta+q^p-q^c \label{dynamicsstoragehot1}\\
\dot{V}^{S_c}&=-B\theta-q^p+q^c,
\end{align}
where $q^p$ and $q^c$ are the flows trough the heat exchanger of the producer and consumer, respectively, and $\theta$ is the flow on a link, which are all given in $m^3/s$. By defining $x={V}^{S_h}$, $u_e=\theta$, $u_p=q^p$ and $d=q^c$, it is easy to see that (\ref{dynamicsstoragehot1}) has the same dynamics as (\ref{model}). Since $\dot{V}^{S_h}+\dot{V}^{S_c}=0$ implies that ${V}^{S_h}(t)+{V}^{S_c}(t)={V}^{S_h}(0)+{V}^{S_c}(0)$ it is trivial to obtain the state of ${V}^{S_c}$, if ${V}^{S_h}$ is given. To this end we perform a simulation where we only consider (\ref{dynamicsstoragehot1}).

\subsection{Simulation}
\begin{figure}
\centering
\includegraphics[trim=3cm 7cm 3cm 7cm, clip=true, width= 0.48 \textwidth]{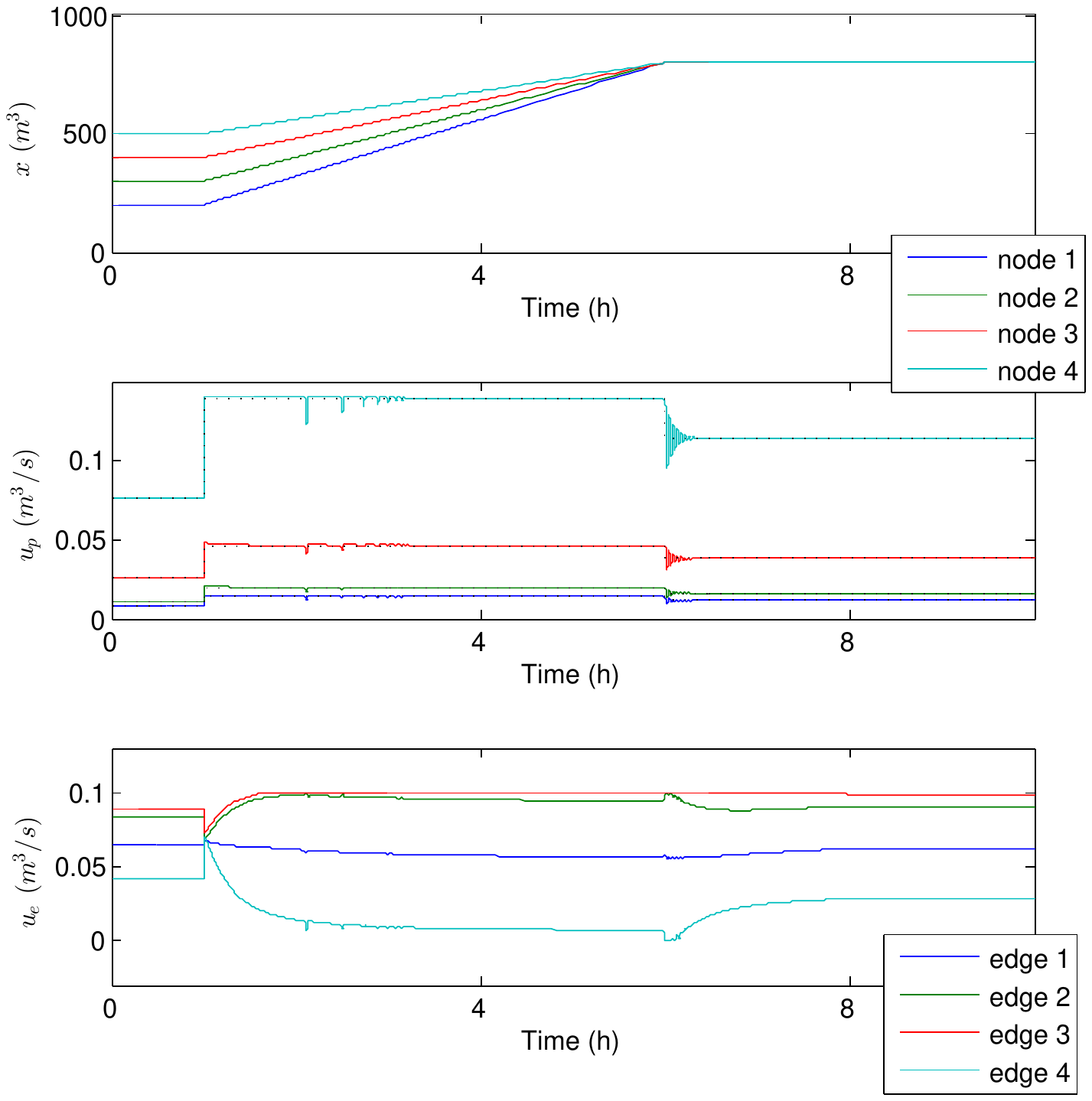}\\
\caption{Volumes, flows and production in the presence of saturation. }\label{fig1}
\end{figure}
\begin{figure}
\centering
\includegraphics[trim=3cm 7.5cm 3cm 7.5cm, clip=true, width= 0.48 \textwidth]{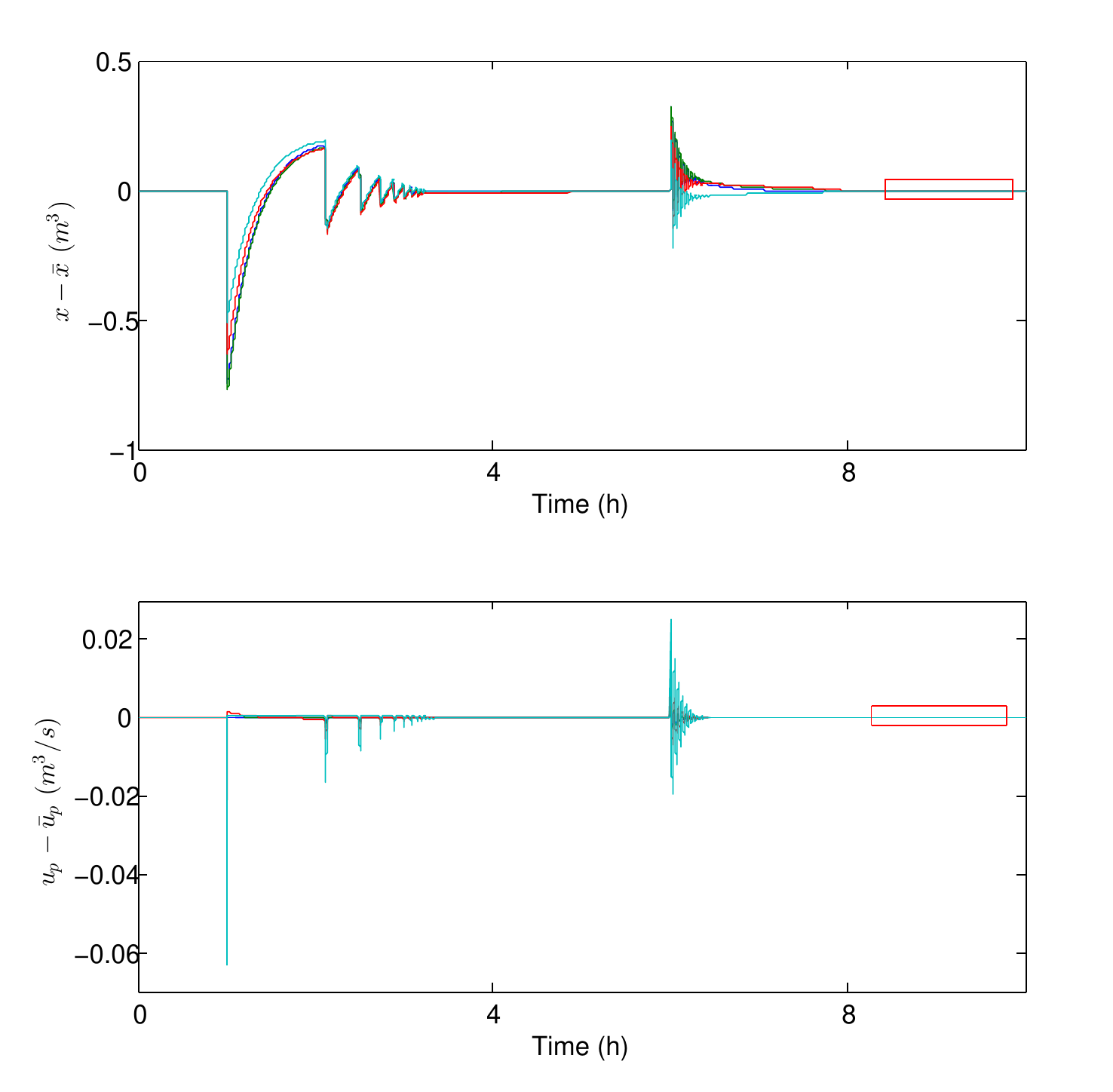}\\
\caption{Deviations from the volume setpoints and optimal production. }\label{figerror}
\end{figure}
\begin{figure}
\centering
\includegraphics[trim=3cm 9.5cm 3cm 9.5cm, clip=true, width= 0.48 \textwidth]{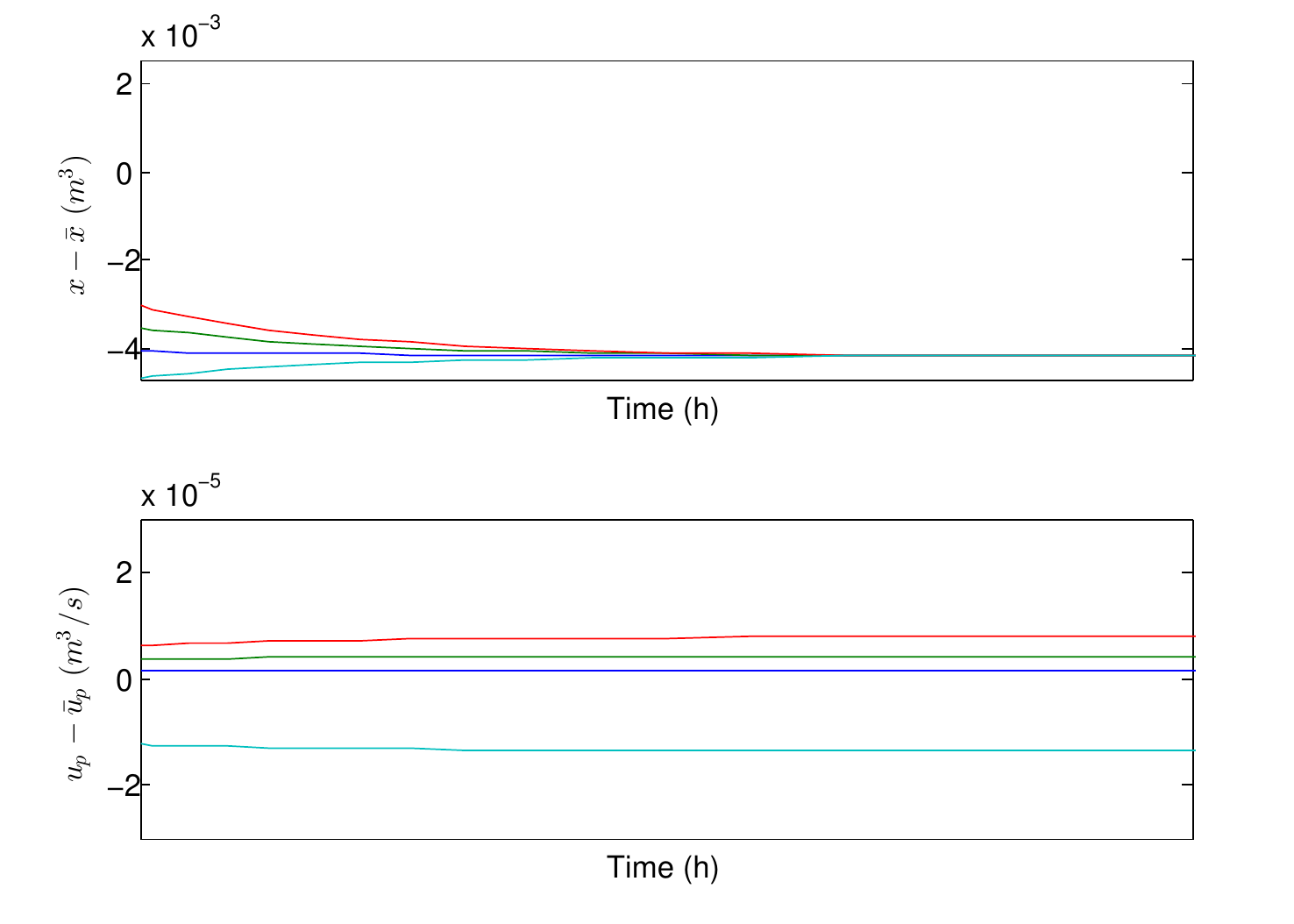}\\
\caption{Enlargement of the highlighted areas of Figure \ref{figerror}. }\label{figerrorlarge}
\end{figure}
We perform a simulation over a $24$ hour time interval and use a circle graph consisting of four nodes. The entries of the quadratic cost functions are given by
\be
Q=\text{diag}\left(
  \begin{array}{cccc}
   1 & 0.7& 0.3& 0.1 \\
  \end{array}
\right),\nonumber
\ee
while $s$ and $r$ are zero vectors. We initialize the system at steady state, with the demand and the volume setpoint given by
\begin{align}
d&=-\left(
  \begin{array}{cccc}
   0.03 &0.03& 0.03& 0.03 \\
  \end{array}
\right)\nonumber \\
\bar{x}&=\left(
  \begin{array}{cccc}
   200 & 300 & 400 & 500\\
  \end{array}
\right).\nonumber
\end{align}
We investigate the response of the system to a ramp reference signal as well as to an increase in demand. First, at $t=1h$ we switch from a constant reference signal to a ramp such that at $t=6h$, $\bar{x}$ becomes
\be
\bar{x}=\left(
  \begin{array}{cccc}
   800 & 800 & 800 & 800\\
  \end{array}
\right).\nonumber
\ee
Soon after this interval we increase the demand by $50\%$ and keep the setpoints constant. The saturation bounds on the production are given by $u_p^-=0m^3/s$ and $u_p^+=0.14m^3/s$ while $u_e^+=0.1m^3/s$ and the error bounds, as defined in (\ref{xgoal2}) and (\ref{u_p_goal2}), are set to $\epsilon_1=10^{-2} $ and $\epsilon_2=10^{-4}$.

Based on Lemma \ref{lemma3} in Appendix \ref{appendix} we can explicitly calculate the bounds on $\gamma_c$, $\gamma_p$ and $\gamma_l$ (see  (\ref{boundgammac}) and (\ref{boundgamma})).
To illustrate how these gains are found, we investigate them for the first interval ({\it i.e.} between $0h$ and $1h$). In that case the numerical value of the right hand side of (\ref{boundgammac}) is $0.1324$  and $\|\Phi^{-1}\bar{Q}^2Q\tilde{d}\|=0.7773$. Furthermore, we have that $\frac{\|\Phi^{-1} \bar{Q}^2 Q \tilde{d}\|}{\|\Phi^{-1}\bar{Q}\| }=   0.0676 $ and by taking $\theta= 0.9985$ this implies that $\delta_\theta=10^{-4}$. Additionally we have that $\delta_p=0.0058 $ and $\delta_e=0.0087$ which means that $\min\left\{\delta_p, \delta_e, \delta_\theta,\epsilon_2 \right\}=\epsilon_2$. It is now easily verified that the conditions in (\ref{boundgammaandgammac}) are satisfied if $\gamma_c< 0.1109$ and $\gamma_p^2/\gamma_l< 1.90\cdot 10^{-4}$. If we therefore take $\gamma_p=0.01$, $\gamma_l=0.53$, $\gamma_e=0.01$ and $\gamma_c=0.11$ the conditions in (\ref{boundgammaandgammac}) are clearly satisfied.



Plots of the resulting simulations can be found in Figure \ref{fig1}, in which we see that in all intervals $\lim_{t\rightarrow\infty} x\approx\bar{x}$ and $\lim_{t\rightarrow\infty} u_p\approx\bar{u}_p$. The optimal production $\bar{u}_p$ in the middle plot is given by the dotted black line from which one can see that the jumps in the reference signal (corresponding to a transition to a charging phase) affects the optimal production levels. In the top plot of Figure \ref{fig1} we can see that $x$ is able to track the piecewise constant reference signal $\bar{x}$. The flow injected by the producers, depicted in the middle plot, show some transient behaviour after the switch to the charging phase and increase of demand. In the interval $1 \leq t \leq 4$, we can also see that the production on node $4$ and the flows on edge $3$ are subject to saturation which cause some wind-up phenomena.

In Figure \ref{figerror} we see in the upper plot the deviation of $x$ from $\bar{x}$ and in the bottom plot the deviation of $u_p$ from $\bar{u}_p$. Again the transient behaviour after $t=1$ and $t=6$ is clearly visible as well as the wind-up phenomena for $1 \leq t \leq 4$. Finally, an enlargement of the highlighted areas in Figure \ref{figerror} can be found in Figure \ref{figerrorlarge}. From this Figure we can clearly see that at the end of the last interval we have that $\|x(t)-\bar{x}(t)\|<\epsilon_1 $ and $\|u_p(t)-\overline{u}_p \|<\epsilon_2$, respectively.

\section{CONCLUSION}
\label{conclusion}
We proposed dynamic feedback controllers that solve a quasi-optimal regulation problem with saturation on the inputs and the flows. The controllers are composed of two parts: the first part regulates the flows on the edges, which results in load balancing while the second part provides an optimal input on the nodes at steady state. We have stated sufficient conditions such that, in spite of the saturations, the controllers are still able to achieve quasi-optimal regulation.

An open problem is to consider general convex cost functions instead of the linear-quadratic cost functions we use. Another interesting problem is to extend this setup to time-varying disturbances as considered in \cite{burger2015dynamic}, or extend it to a tracking problem of more general time varying signals. Lastly, we would like to investigate the existence of alternative controllers, that guarantee asymptotic convergence to the optimal steady state in the constrained case.

\section*{ACKNOWLEDGMENT}
The work of C.~De Persis, P.~Tesi and T.W.~Scholten is supported by the research grant {\it Flexiheat} (Ministerie van Economische Zaken, Landbouw en Innovatie). The work of  C.~De Persis is also supported by  {\it Efficient Distribution of Green Energy} (Danish Research Council of Strategic Research) and {\it QUICK} (The Netherlands Organization of Scientific Research). Furthermore, the authors would like to thank JieQiang Wei for his valuable feedback on the paper. 


\bibliographystyle{plain}        
\bibliography{overallbib9}           



\appendix
\section{LEMMAS}
\label{appendix}

In order to prove Theorem \ref{maintheorem} we introduce the following lemmas.
\begin{lemm}
\label{lemmaAcolumns}
Let $Q$ be a diagonal matrix with positive entries, let $\tilde{Q}$ and $\bar{Q}$ be as in (\ref{Qtilde}) and (\ref{Qbar}), respectively and let $L_c$ be an undirected strongly connected Laplacian matrix, then
$(\gamma \bar{Q}-L_cQ+\mathds{1}\mathds{1}^T)$, $(\gamma \bar{Q}-L_cQ+\frac{1}{n}\mathds{1}\mathds{1}^T)$ and $\tilde{Q}^T\tilde{Q}$ are full rank for all $\gamma\in\mathbb{R}_{\geq0}$.
\end{lemm}

\begin{proof}
We will first proof that all the columns of $\tilde{Q}$ are linearly independent for all $\gamma\in\mathbb{R}_{\geq0}$. From this we will then conclude that $\tilde{Q}^T\tilde{Q}$, $(\gamma \bar{Q}-L_cQ+\mathds{1}\mathds{1}^T)$ and $(\gamma \bar{Q}-L_cQ+\frac{1}{n}\mathds{1}\mathds{1}^T)$ are full rank. Let
\be
A_n :=L_cQ-\gamma\bar{Q}, \label{A_n}
\ee
now, since $\bar{Q}_{ij}\geq0$ and $(L_cQ)_{ij}\leq0$ for all $i \neq j$ we know that the off-diagonal elements of $A_n$ are non-positive. Furthermore, since the graph associated to $L_c$ is strongly connected, we have that $(L_c)_{ii}>0$ and since $\mathds{1}^T(L_cQ-\gamma\bar{Q})=0$, all the diagonal elements of $A_n$ are strictly positive. Therefore we can write
\be
A_n=\left(
    \begin{array}{ccccc}
      a_{11} & -a_{12}  & \dots & -a_{1n} \\
      -a_{21} & a_{22} & \dots& -a_{2n} \\
      \vdots & \vdots  & \ddots & \vdots \\
      -a_{n1} & -a_{n2}  & \dots & a_{nn} \\
    \end{array}
  \right),
\ee
with $a_{ii}>0$ for all $i$ and $a_{ij}\geq0$ for all $i\neq j$. Moreover, since
\be
\mathds{1}^T (\gamma \bar{Q}-L_cQ)=0,
\ee
we can conclude that the diagonal elements are equal to the negative column sum of the off diagonal elements, {\it i.e.} $a_{ii}=\sum_{k=1, k\neq i}^{n}a_{ki}$. We will now prove that $\tilde{Q}$ is full column rank. To this end we consider a square sub-matrix of $\tilde{Q}$ which we define as
\be
\tilde{Q}_{sub}=\left(
    \begin{array}{cc}
      A_{n-1} & -a_{[n]}\\
      \mathds{1}_{n-1}^T & 1\\
    \end{array}
  \right),
\ee
where $a_{[n]}:=\left(
                  \begin{array}{cccc}
                    a_{1n} & a_{2n} & \dots & a_{(n-1)n} \\
                  \end{array}
                \right)^T$. By the Schur complement we know that
\be
\begin{aligned}
\det(\bar{Q}_{sub})=&\det( A_{n-1} +a_{[n]}\mathds{1}_{n-1}^T )\\
=& (1+ \mathds{1}_{n-1}^T A_{n-1}^{-1}a_{[n]})\det(A_{n-1}).
\end{aligned}
\ee
Since $A_{n-1}$ is a diagonal column-dominant matrix we obtain from the Gershgorin circle theorem that all the eigenvalues of $A_{n-1}^T$ are strictly positive. This implies that $\det(A_{n-1}^T)\neq0$ and since $A_{n-1}$ is square we obtain $\det(A_{n-1})\neq0$. Furthermore, again due to the diagonal dominance property of $A_{n-1}$, we have that every principal minor (see {\it e.g.} \cite{fiedler1962matrices} for a definition) of $A_{n-1}$ is positive. This implies that  $A_{n-1}$ is inverse-positive, as is proven in \cite{fiedler1962matrices}. From this it follows that $\mathds{1}_{n-1}^T A_{n-1}^{-1}a_{[n]}>0$ which results in $\det(\tilde{Q}_{sub})\neq0$, implying that all the columns of $\tilde{Q}_{sub}$ are linearly independent. Since the number of columns of $\tilde{Q}$ and $\tilde{Q}_{sub}$ are equal, we can conclude that all the columns of $\tilde{Q}$ are linearly independent for all $\gamma\in\mathbb{R}_{\geq0}$. This, and since $\tilde{Q}^T\tilde{Q}$ is a square matrix, immediately implies that $\tilde{Q}^T\tilde{Q}$ is full rank. Next we use the following identity
\be
\begin{aligned}
  \tilde{Q}^T\tilde{Q} =&A_n^TA_n+\mathds{1}\mathds{1}^T\\
   =&\left(\mathds{1}\mathds{1}^T-A_n\right)^T\left(\frac{1}{n}\mathds{1}\mathds{1}^T-A_n\right)
   \end{aligned}\label{fullrankmatrices}
\ee
where $A_n$ is as in (\ref{A_n}). Since $\tilde{Q}^T\tilde{Q}$ is full rank, it follows directly that also $(\gamma \bar{Q}-L_cQ+\frac{1}{n}\mathds{1}\mathds{1}^T)$ is full rank. In fact, suppose it is not full rank, then there exists a $x\neq0$ such that $(\gamma \bar{Q}-L_cQ+\frac{1}{n}\mathds{1}\mathds{1}^T)x=0$. Due to (\ref{fullrankmatrices}) and (\ref{A_n}) this implies that $\tilde{Q}^T\tilde{Q}x=0$, which is a contradiction with $\tilde{Q}^T\tilde{Q}$ being full rank. Using the same argumentation it follows directly from (\ref{fullrankmatrices}) that also $(\gamma \bar{Q}-L_cQ+\mathds{1}\mathds{1}^T)$ is full rank.
\end{proof}

\begin{lemm}
Let $\bar{Q}$ and $\Phi$ as in (\ref{Qbar}) and (\ref{Phi}), then $\Phi^{-1}$ exists. Furthermore, let
\be
0<\gamma\leq\frac{\theta}{\|\Phi^{-1}\bar{Q}\|},\label{gammatheta}
\ee
for some $0<\theta<1$, and let $\hat{u}_p$, $\hat{u}_e$ and $\hat{x}$ be as in (\ref{uphat}), (\ref{upehat}) and (\ref{hatxstarsolved}), then
\begin{align}
\|\hat{u}_p\|\leq& \gamma \frac{1}{1-\theta} \|\Phi^{-1} \bar{Q}^2 Q \tilde{d}\| \label{boundhatup} \\
\|\hat{u}_e\|\leq& \gamma \frac{1}{1-\theta} \|B^\dagger\| \cdot \|\Phi^{-1} \bar{Q}^2 Q \tilde{d}\| \label{boundBhatue}\\
\begin{split}\|\hat{x}\|\leq&\frac{\gamma_c}{\mathds{1}^TQ^{-1}\mathds{1}}\left(  \| \mathds{1}\mathds{1}^TQ^{-1} \bar{Q} Q \tilde{d}\|\right.\\
&+\frac{\gamma}{1-\theta}\| \mathds{1}\mathds{1}^TQ^{-1}\| \cdot \left.\|\Phi^{-1} \bar{Q}^2 Q \tilde{d}\|\right).
\end{split}\label{boundhatx}
\end{align}
\label{lemma2}
\end{lemm}
\begin{proof}
First we prove that $\Phi^{-1}$ exists. From Lemma \ref{lemmaAcolumns} it follows directly that $(\gamma \bar{Q}+\Phi)$ is invertible for any $\gamma\geq0$. By taking $\gamma=0$ it follows that $\Phi^{-1}$ exists. Next we prove that (\ref{boundhatup})-(\ref{boundhatx}) holds.
To do this we make use of the following identity
\be
\begin{aligned}
&\sum_{k=0}^N(-\gamma\Phi^{-1}\bar{Q})^k(I+\gamma \Phi^{-1}\bar{Q}) \\
=&I+(-1)^N(\gamma \Phi^{-1}\bar{Q})^{N+1}.
\end{aligned}\label{sumN}
\ee
Due to (\ref{gammatheta}) we have
\be
\|\gamma\Phi^{-1}\bar{Q}\|<1, \label{normlesthenone}
\ee
which implies that $I+\gamma\Phi^{-1}\bar{Q}$ is invertible. That is, suppose that $I+\gamma\Phi^{-1}\bar{Q}$ is not invertible, then there exists a non-zero $x$ such that $(I+\gamma\Phi^{-1}\bar{Q})x=0$. In such a case $0\leq \|x\|(1 - \|\gamma\Phi^{-1}\bar{Q}\|) $, which contradicts (\ref{normlesthenone}).
%
Then, after lengthy but standard arguments, (\ref{sumN}) and (\ref{normlesthenone}) imply that
\be
\sum_{k=0}^\infty(-\gamma\Phi^{-1}\bar{Q})^k=(I+\gamma \Phi^{-1}\bar{Q})^{-1}, \label{infsum}
\ee
from which we obtain, together with (\ref{gammatheta}), that
\be
\begin{aligned}
\|(I+\gamma \Phi^{-1}\bar{Q})^{-1}\|\leq&\sum_{k=0}^\infty\|(-\gamma\Phi^{-1}\bar{Q})\|^k\\
\leq&\sum_{k=0}^\infty\theta^k.
\end{aligned}\label{infsum2}
\ee
Notice that the right  hand side of (\ref{infsum2}) is a standard geometric series and this implies that
\be
\|\left(\gamma\Phi^{-1}\bar{Q}+I\right)^{-1}\|\leq\frac{1}{1-\theta}.  \label{bound1}
\ee
Combining (\ref{bound1}) with (\ref{uphat}) and (\ref{hatxp}) gives us that
\begin{align}
\|\hat{u}_p\|=&\|-\gamma \left(\gamma\bar{Q}+\Phi\right)^{-1} \bar{Q}^2 Q   \tilde{d} \|\\
\leq&\gamma\|\left(\gamma\Phi^{-1}\bar{Q}+I\right)^{-1}\| \cdot \|\Phi^{-1} \bar{Q}^2 Q \tilde{d}\|  \\
\leq&\gamma\frac{1}{1-\theta}\|\Phi^{-1} \bar{Q}^2 Q \tilde{d}\|,
\end{align}
which proves (\ref{boundhatup}). Similarly, combining (\ref{bound1}) with (\ref{upehat}) and (\ref{hatxe})
give us (\ref{boundBhatue}).
Finally, again using (\ref{bound1}) and combining this with (\ref{hatx}), we obtain
\be
\begin{aligned}
\|\hat{x}\|=&\| -\gamma_c\frac{\mathds{1}\mathds{1}^TQ^{-1}}{\mathds{1}^TQ^{-1}\mathds{1}} \cdot \\
&\left(I- \gamma\left(\gamma\bar{Q}+\Phi\right)^{-1}\bar{Q} \right)  \bar{Q} Q  \tilde{d}\|\\
\leq& \frac{\gamma\gamma_c}{\mathds{1}^TQ^{-1}\mathds{1}}\| \mathds{1}\mathds{1}^TQ^{-1}\left(\gamma\bar{Q}+\Phi\right)^{-1}  \bar{Q}^2 Q  \tilde{d}\|\\
&+\frac{\gamma_c}{\mathds{1}^TQ^{-1}\mathds{1}}\| \mathds{1}\mathds{1}^TQ^{-1}  \bar{Q} Q \tilde{d}\|\\
\leq&\frac{\gamma\gamma_c}{\mathds{1}^TQ^{-1}\mathds{1}}\frac{1}{1-\theta}\| \mathds{1}\mathds{1}^TQ^{-1}\| \cdot \|\Phi^{-1}  \bar{Q}^2 Q \tilde{d}\|\\
&+ \frac{\gamma_c}{\mathds{1}^TQ^{-1}\mathds{1}}\| \mathds{1}\mathds{1}^TQ^{-1}  \bar{Q} Q \tilde{d}\|,
\end{aligned}
\ee
which implies (\ref{boundhatx}) and concludes the proof.
\end{proof}

\begin{lemm}\label{lemma3}
Let $\hat{u}_p$, $\hat{u}_e$ be as in (\ref{uphat}), (\ref{upehat}) and let $\hat{x}$ be as in (\ref{xstar}) and (\ref{hatxstarsolved}). If $\gamma_c$, $\gamma_p$ and $\gamma_l$ are such that for $0<\theta<1$,
\begin{subequations}
\begin{align}
 &\gamma_c<\frac{\mathds{1}^TQ^{-1}\mathds{1}\epsilon_1 } {\| \mathds{1}\mathds{1}^TQ^{-1} \bar{Q} Q \tilde{d}\| + \| \mathds{1}\mathds{1}^TQ^{-1}\| \epsilon_2}\label{boundgammac}\\
  &\|\Phi^{-1} \bar{Q}^2 Q \tilde{d}\|  \frac{\gamma_p^2}{\gamma_l} < \frac{1-\theta}{\gamma_c}\min\left\{\delta_p, \delta_e, \delta_\theta,\epsilon_2 \right\},\label{boundgamma}
\end{align}\label{boundgammaandgammac}
\end{subequations}
with $\tilde{d}$, $\bar{Q}$ and $\Phi$ as in (\ref{tilded})-(\ref{Phi}) and
\begin{align}
 \delta_p =&\min\{\min_{i}\{
(u_p^+-\bar{u}_p)_i\},\min_{j}\{
(\bar{u}_p-u_p^-)_j\}\}\label{deltap}\\
 \delta_e=& \frac{1}{\left|\left|B^\dagger\right|\right|}
 \min\{\min_{i}\{(u_e^+-\bar{u}_e)_i\},\min_{j}\{(\bar{u}_e)_j\}\} \label{deltae}\\
 \delta_\theta=&\left\{\begin{array}{lr} \frac{\|\Phi^{-1} \bar{Q}^2 Q \tilde{d}\|}{\|\Phi^{-1}\bar{Q}\| }\frac{\theta}{(1-\theta) } & \text{if } \|\Phi^{-1}\bar{Q}\|\neq0 \\ +\infty & \text{if } \|\Phi^{-1}\bar{Q}\|=0 \end{array}
 \right.\label{delta_theta}.
\end{align}
then
\begin{align}
\begin{split}
\|\hat{u}_p\|<& \min\{\min_{i}\{(u_p^+-\bar{u}_p)_i\},\\
& \quad \quad \min_{j}\{(\bar{u}_p-u_p^-)_j\},\epsilon_2\}
\end{split}\label{boundonhatup} \\
\|\hat{x}\|<&\text{ } \epsilon_1 \label{boundonhatx}\\
\|\hat{u}_e\|< & \min\{\min_{i}\{(u_e^+-\bar{u}_e)_i\},\min_{j}\{(\bar{u}_e)_j\}\}.\label{boundonhatue}
\end{align}
\end{lemm}
\begin{proof}
In order to prove this, we make use of Lemma \ref{lemma2}, where we note that (\ref{gammatheta}) is satisfied due to (\ref{boundgamma}). To prove (\ref{boundonhatup}) we combine (\ref{boundhatup}) and (\ref{boundgamma}) such that
\begin{align}
\|\hat{u}_p\|\leq& \gamma \frac{1}{1-\theta} \left|\left|\Phi^{-1} \bar{Q}^2 Q \tilde{d}\right|\right| \\
\begin{split}
<&\min\{\min_{i}\{(u_p^+-\bar{u}_p)_i\}, \\
&\quad \quad \min_{j}\{(\bar{u}_p-u_p^-)_j\},\epsilon_2\}.
\end{split}\label{hatpboundepsilon2}
\end{align}
Using  (\ref{hatpboundepsilon2}) with (\ref{boundgammac}) and (\ref{boundhatx}) gives us
\begin{align}
\|\hat{x}\|\leq&\frac{\gamma_c}{\mathds{1}^TQ^{-1}\mathds{1}}\left(\| \mathds{1}\mathds{1}^TQ^{-1} \bar{Q} Q \tilde{d}\|\right. \nonumber\\
&\quad +\left.\| \mathds{1}\mathds{1}^TQ^{-1}\|\epsilon_2\right)\\
<&\epsilon_1, \nonumber
\end{align}
which implies (\ref{boundonhatx}). Lastly, from (\ref{boundgamma}) and (\ref{boundBhatue}) we have that
\be\label{boundBhatue2}
\begin{aligned}
\|\hat{u}_e\|<&\left|\left|B^\dagger\right|\right|\delta_e\\
\leq& \min\{\min_{i}\{(u_e^+-\bar{u}_e)_i\},\min_{j}\{(\bar{u}_e)_j\}\},
\end{aligned}
\ee
 with $\delta_e$ as in (\ref{deltae}). This implies (\ref{boundonhatue}) and concludes the proof.
\end{proof}

\begin{rema}\label{remarkgains}
To guarantee that Theorem \ref{maintheorem} solves Problem \ref{problem2}, a sufficient condition for $\gamma_c>0$, $\gamma_p>0$ and $\gamma_l>0$ is that they satisfy (\ref{boundgammaandgammac}). Note that these gains can always be found since $\delta_e$, $\delta_p$ and $\delta_\theta$ are all strictly positive due to the feasibility condition. Moreover, in the special case that all nodes supply their own demand ({\it i.e.,} if $\bar{Q} Q \tilde{d}=0$), (\ref{boundgamma}) is satisfied for any $\gamma_p$ and $\gamma_l$. Since $\gamma_c$ acts as the proportional feedback in (\ref{contesat}) and has to be chosen sufficiently small due to (\ref{boundgammac}), a smaller steady state error comes at the cost of a lower convergence rate. Although the controller is fully distributed, global information of the topology, cost functions, disturbance bounds and saturation bounds are required to guarantee bounds on the deviation from the optimal steady state. It is easy to show that a $\gamma_c>0$, $\gamma_p>0$ and $\gamma_l>0$ can be found such that (\ref{boundgammaandgammac}) is satisfied for all the disturbances whose magnitude belongs to a compact interval of values.
\end{rema}

\begin{lemm}
Let ${\bar{x}}_p$ be as in (\ref{barx_def}) and let $\bar{x}_e$, ${\hat{x}}_p$ and $\hat{x}_e$ be the solutions to (\ref{Bbarxe}) and (\ref{hatsss}). If all the conditions of Lemma \ref{lemma3} are satisfied,
then
\begin{align}
x_e^-&<0 \quad \quad x_e^+>0 \label{sateboundsposmin}\\
x_p^-&<0 \quad \quad x_p^+>0, \label{satpboundsposmin}
\end{align}
with $x_e^-$, $x_e^+$, $x_p^-$ and $x_p^+$ as defined in
(\ref{xemindef})-(\ref{xpmaxdef}).
\label{bounds_not_sat}
\end{lemm}
\begin{proof}
From Lemma \ref{lemma3} we get that
\be
\|\hat{u}_p\|< \min\{\min_{i}\{(u_p^+-\bar{u}_p)_i\},\min_{j}\{(\bar{u}_p-u_p^-)_j\},
\ee
and
\be
\|\hat{u}_e\|< \min\{\min_{i}\{(u_e^+-\bar{u}_e)_i\},\min_{j}\{(\bar{u}_e)_j\}\}.
\ee
This, together with (\ref{ineqmachtingcondition}) and (\ref{u_e^+min}), implies that
\be
u_p^- < \hat{u}_p+\bar{u}_p < u_p^+-\bar{u}_p
\ee
\be
0 < \hat{u}_e+\bar{u}_e < u_e^+-\bar{u}_e,
\ee
and due to (\ref{updef}) and (\ref{upedef}) we get
\begin{align}
u_p^-<&\gamma_pQ^{-1}(\bar{x}_p+\hat{x}_p)-r< u_p^+\\
0<&-\gamma_e(\bar{x}_e+\hat{x}_e)< u_e^+.
\end{align}
In light of (\ref{xemindef})-(\ref{xpmaxdef}) we can conclude that (\ref{sateboundsposmin}) and (\ref{satpboundsposmin}) are satisfied, which concludes the proof.
\end{proof}

\begin{lemm} Let all the conditions of Theorem \ref{maintheorem} be satisfied. Given the Lyapunov function
\begin{align}
V(\tilde{x},\tilde{x}_e,\tilde{x}_p)=& \medspace \frac{1}{2}\|\tilde{x}\|^2+\sum_{i=1}^n S^p_i +\sum_{i=1}^m S^e_i, \label{lyapunovfunction}
\end{align}
where
\be
\begin{aligned}
S^p_i:=\int\limits_0^{(\tilde{x}_p)_i} &\emph{sat}( y,(\frac{1}{\gamma_p}(Qu_p^-+r)-({\bar{x}}_p+{\hat{x}}_p))_i,\\
&(\frac{1}{\gamma_p}(Qu_p^++r)-({\bar{x}}_p+{\hat{x}}_p))_i)dy,
\end{aligned} \label{Sp}
\ee
and
\be
\begin{aligned}
S^e_i:=\frac{1}{\gamma_e^2}\int\limits_0^{-\chi_i}&\emph{sat}(y,(-\gamma_e(\bar{x}_e+{\hat{x}}_e)))_i,\\
&(u_e^+-\gamma_e(\bar{x}_e+{\hat{x}}_e)))_i)dy,
\end{aligned} \label{Se}
\ee
with $\chi=\gamma_e\tilde{x}_e+\gamma_cB^T\tilde{x}$, then
\be
\dot{V}(\tilde{x},\tilde{x}_e,\tilde{x}_p)\leq0,\label{vdotleqzerogoal}
\ee
and the set
\be
\begin{aligned}
\mathcal{Q}=&\{(\tilde{x},\tilde{x}_e,\tilde{x}_p)|V(\tilde{x},\tilde{x}_e,\tilde{x}_p)\leq D)\},
\end{aligned}
\label{set_S}
\ee
with $D\geq0$, is nonempty, compact and forward invariant for system (\ref{sat_dynamics}).
\label{lemma1}
\end{lemm}

\noindent\hspace{2em}{\itshape Proof \footnote{This proof is an extension of a proof presented in \cite{wei2013load}. The proof in that paper does not consider the dynamics of $x_p$ nor an input at the node with associated cost function, {\it i.e. }$x_p=0$, $Q=0$ and $r=0$.}: }
We first prove (\ref{vdotleqzerogoal}), then we will show that $\mathcal{Q}$ is forward invariant and finally we prove that $\mathcal{Q}$ is compact and non-empty. By evaluating the partial derivatives of (\ref{lyapunovfunction}), we see that
 \begin{equation}
\begin{aligned}
 \frac{\partial V}{\partial\tilde{x}}&=\tilde{x}^T  -\gamma_c\text{sat}_e(\tilde{x},\tilde{x}_e)^TB^T \\
\frac{\partial V}{\partial\tilde{x}_p}&=\text{sat}_p(\tilde{x}_p)^T\\
\frac{\partial V}{\partial\tilde{x}_e}&=-\frac{1}{\gamma_e}\text{sat}_{ e }(\tilde{x},\tilde{x}_e)^T,
\end{aligned}
\end{equation}
with $\text{sat}_e(\tilde{x},\tilde{x}_e)$ and $\text{sat}_p(\tilde{x}_p)$ as defined in (\ref{sate}) and (\ref{satp}), respectively. Hence, with the help of Lemma \ref{lemma_closedloop}, it is easy to see that
\be
\dot{V}= -\gamma_c\|B\text{sat}_e(\tilde{x},\tilde{x}_e)\|^2-\gamma_l\|B_c^T\text{sat}_p(\tilde{x}_p)\|^2 \label{vdotresultsat}
\ee
where $B_c$ is the incidence matrix associated to the communication graph. From (\ref{vdotresultsat}) it is easy to see that (\ref{vdotleqzerogoal}) is satisfied, which directly implies that $\mathcal{Q}$ is forward invariant.

Finally we will prove that (\ref{set_S}) is compact. Note that this is equivalent to $\mathcal{S}$ being closed and bounded. From the definition of $\mathcal{S}$ it follows trivially that it closed which leaves us with the proof that (\ref{set_S}) is bounded.

By Lemma \ref{bounds_not_sat} we know that there exists an open ball that contains the origin that lies within the bounds of the saturation functions in (\ref{Sp}) and (\ref{Se}). Notice that this implies that
$ S^p_i\geq0$ and $ S^e_j\geq0$ for all $i$ and $j$. Now suppose that $|\tilde{x}_i|\rightarrow \infty$, then necessarily $V(\tilde{x},\tilde{x}_e,\tilde{x}_p)\rightarrow \infty$, however this is in contradiction with (\ref{vdotresultsat}) implying that $\tilde{x}$ is bounded. Now suppose that $|(\tilde{x}_p)_i|\rightarrow \infty$, then necessarily $ S^p_i \rightarrow \infty$ due to (\ref{satpboundsposmin}). This implies again that $V(\tilde{x},\tilde{x}_e,\tilde{x}_p)\rightarrow \infty$ from which we can conclude that $\tilde{x}_p$ is bounded. Lastly we prove that $\tilde{x}_e$ is bounded. Suppose that
$|(\tilde{x}_e)_i|\rightarrow \infty$ then also $|-(\gamma_e(\tilde{x}_e)+\gamma_cB\tilde{x})_i|\rightarrow \infty$ since $\tilde{x}$ is bounded. This, together with (\ref{sateboundsposmin}) implies that $S^e_j \rightarrow \infty$. Therefore also $\tilde{x}_e$ is bounded and we can therefore conclude that $\mathcal{Q}$ is compact.
Lastly we prove that $\mathcal{Q}$ is non-empty. Note that $V(0,0,0)=0$, this implies that the origin is contained in $\mathcal{Q}$, which concludes the proof.
\endproof
\end{document}